\documentclass[fleqn,reqno,11pt,a4paper,final]{amsart}

\usepackage[a4paper,left=30mm,right=30mm,top=30mm,bottom=30mm,marginpar=23mm]{geometry} 
\usepackage{amsmath}
\usepackage{amssymb}
\usepackage{amsthm}
\usepackage{amscd}
\usepackage[ansinew]{inputenc}
\usepackage{cite}
\usepackage{bbm}
\usepackage{color}
\usepackage[english=american]{csquotes}
\usepackage[final]{graphicx}
\usepackage{hyperref}
\usepackage{calc}
\usepackage{mathptmx}

\graphicspath{{../Pictures/}}

\numberwithin{equation}{section}

\newtheoremstyle{thmlemcorr}{10pt}{10pt}{\itshape}{}{\bfseries}{.}{10pt}{{\thmname{#1}\thmnumber{ #2}\thmnote{ (#3)}}}
\newtheoremstyle{thmlemcorr*}{10pt}{10pt}{\itshape}{}{\bfseries}{.}\newline{{\thmname{#1}\thmnumber{ #2}\thmnote{ (#3)}}}
\newtheoremstyle{remexample}{10pt}{10pt}{}{}{\bfseries}{.}{10pt}{{\thmname{#1}\thmnumber{ #2}\thmnote{ (#3)}}}
\newtheoremstyle{ass}{10pt}{10pt}{}{}{\bfseries}{.}{10pt}{{\thmname{#1}\thmnumber{ A#2}\thmnote{ (#3)}}}

\theoremstyle{thmlemcorr}
\newtheorem{theorem}{Theorem}
\numberwithin{theorem}{section}
\newtheorem{lemma}[theorem]{Lemma}

\newtheorem{proposition}[theorem]{Proposition}

\newtheorem{definition}[theorem]{Definition}

\theoremstyle{plain}
\newtheorem*{theorem*}{Theorem}
\newtheorem*{lemma*}{Lemma}

\theoremstyle{thmlemcorr*}
\newtheorem{corollary*}[theorem]{Corollary}
\newtheorem{proposition*}[theorem]{Proposition}
\newtheorem{problem*}[theorem]{Problem}
\newtheorem{conjecture*}[theorem]{Conjecture}
\newtheorem{definition*}[theorem]{Definition}

\theoremstyle{remexample}
\newtheorem{remark}[theorem]{Remark}

\theoremstyle{ass}

\newcommand{\Lrm}{\mathrm{L}}

\newcommand{\Rbb}{\mathbb{R}}

\DeclareMathOperator{\curl}{curl}
\DeclareMathOperator{\dist}{dist}
\DeclareMathOperator{\rank}{rank}

\DeclareMathOperator{\supp}{supp}

\DeclareMathOperator{\Div}{div}

\newcommand{\norm}[1]{\|#1\|}

\newcommand{\R}{\mathbb{R}}

\newcommand{\sbullet}{\begin{picture}(1,1)(-0.5,-2.5)\circle*{2}\end{picture}}
\newcommand{\frarg}{\,\sbullet\,}

\newcommand{\term}[1]{\textbf{#1}}

\def\XXint#1#2#3{{\setbox0=\hbox{$#1{#2#3}{\int}$} 
\vcenter{\hbox{$#2#3$}}\kern-.5\wd0}}

\renewcommand{\epsilon}{\varepsilon}
\renewcommand{\phi}{\varphi}

\begin{document}


\title[Failure of the Chain Rule for the Divergence of Bounded Vector Fields]{Failure of the Chain Rule for the Divergence\\ of Bounded Vector Fields}

\author{Gianluca Crippa}
\address{\textit{Gianluca Crippa:} Departement Mathematik und Informatik, Universit\"at
  Basel, Rheinsprung 21, CH-4051 Basel, Switzerland}
\email{gianluca.crippa@unibas.ch}

\author{Nikolay Gusev}
\address{\textit{Nikolay Gusev:} Dybenko st., 22/3, 94, 125475 Moscow, Russia}
\email{n.a.gusev@gmail.com}

\author{Stefano Spirito}
\address{\textit{Stefano Spirito:} GSSI -- Gran Sasso Science Institute, Viale Francesco Crispi 7, 67100 L'Aquila, Italy}
\email{stefano.spirito@gssi.infn.it}

\author{Emil Wiedemann}
\address{\textit{Emil Wiedemann:} Hausdorff Center for Mathematics and Mathematical Institute, Universit\"at Bonn, Endenicher Allee 60, 53115 Bonn, Germany}
\email{emil.wiedemann@hcm.uni-bonn.de}


\hypersetup{
  pdfauthor = {Gianluca Crippa and Nikolay Gusev and Stefano Spirito and Emil Wiedemann},
  pdftitle = {...},
  pdfsubject = {MSC (2010): ? (primary); ?},
  pdfkeywords = {}
}


\maketitle


\begin{abstract}
We provide a vast class of counterexamples to the chain rule for the divergence of bounded vector fields in three space dimensions. Our convex integration approach allows us to produce renormalization defects of various kinds, which in a sense quantify the breakdown of the chain rule. For instance, we can construct defects which are absolutely continuous with respect to Lebesgue measure, or defects which are not even measures.   
\vspace{4pt}



\vspace{4pt}

\end{abstract}

%

\section{Introduction}
In this paper we consider the classical problem of the chain rule for the divergence of a bounded vector field. Specifically, the problem can be stated in the following way:\\
\\
{\em Let $\Omega\subset\Rbb^d$ be a domain with Lipschitz boundary. Given a bounded vector field $v:\Omega\to\R^d$ tangent to the boundary and a bounded scalar function $\rho:\Omega\to\R$, one asks whether is possible to express the quantity $\Div(\beta(\rho)v)$, where $\beta$ is a smooth scalar function, only in terms of $\beta$, $\rho$ and the quantities $\lambda=\Div v$ and $\nu=\Div(\rho v)$.}\\
\\
Indeed, formally we should have that 
\begin{equation}\label{eq:1}
\Div(\beta(\rho)v)=(\beta(\rho)-\rho\beta'(\rho))\mu+\beta'(\rho)\lambda.
\end{equation}
However, the extension of \eqref{eq:1} to a nonsmooth setting is far from trivial. The chain rule problem is particularly important in view of its applications to the uniqueness and compactness of transport and continuity equations, whose analysis is nowadays a fundamental tool in the study of various equations arising in mathematical physics. Indeed continuity equations arise naturally for instance in compressible fluid mechanics in order to model the evolution of the density of a fluid. 

The chain rule problem for nonsmooth vector fields has been considered in several papers, in particular in \cite{ADM}. There, it is proved that if $v$ is of bounded variation and $\Div(\rho v)$ is a measure, then $\Div(\beta(\rho)v)$ is also a measure and in particular formula \eqref{eq:1} holds for the absolutely continuous parts of $\lambda$ and $\mu$. The singular part is partially characterized in the cited article.
 
In this paper, we prove that in the three dimensional case for vector fields which are merely bounded the formula \eqref{eq:1} is invalid in a very strong sense. Specifically, for a strongly convex function $\beta:(0,\infty)\to\Rbb$ and a given \term{renormalization defect} $f:\Omega\to\Rbb$ we construct a divergence-free vector field $v$ and a scalar function $\rho$ satisfying
\begin{equation}\label{continuity}
\begin{aligned}
\Div(\rho v)&=0\textrm{ in }\Omega\\
\Div(v)&=0\textrm{ in }\Omega\\
v\cdot n&=0\hspace{0.2cm}\text{on $\partial\Omega$}
\end{aligned}
\end{equation}
such that 
\begin{equation}\label{defect}
\Div(\beta(\rho)v)=f.
\end{equation}
More precisely, our main theorem reads as follows:
\begin{theorem}\label{mainthm}
Let $\Omega\subset\Rbb^3$ be a (not necessarily bounded) domain with Lipschitz boundary and $\beta:(0,\infty)\to\Rbb$ be strongly convex. Let moreover $f$ be a distribution such that the equation 
\begin{equation*}
\Div w=f 
\end{equation*}
admits a bounded continuous solution on $\Omega$. 
Then there exist $v\in \Lrm^{\infty}(\Omega;\Rbb^3)$ and $\rho\in \Lrm^{\infty}(\Omega;\Rbb)$ positive and bounded away from $0$ such that~\eqref{continuity} and~\eqref{defect} are satisfied in the sense of distributions. 
\end{theorem}
\begin{remark}
We want to point out that the requirement on $f$ is satisfied for instance when $\Omega$ is bounded and $f\in L^{p}(\Omega)$ with $p>3$. However, there exist also distributions $f$ which are not measures for which the divergence equation admits a bounded continuous solution. In particular, our result shows that if we drop the $BV$ regularity assumption on the vector field $v$, then the quantity $\Div(\beta(\rho) v)$ can fail to be a measure, even though $\lambda$ and $\mu$ vanish.
\end{remark}
\begin{remark}
The theorem is still valid in dimensions higher than $3$, with essentially the same proof.
\end{remark}
It is worth pointing out that in Theorem \ref{mainthm} it is crucial that $d\geq 3$. Indeed, for bounded two dimensional vector fields and strictly positive density $\rho$ bounded away from $0$, formula \eqref{eq:1} has been established in \cite{BG}. Our result can thus be interpreted as complementary to the one in \cite{BG}. 

As mentioned above, the chain rule is strongly connected with the uniqueness problem for transport and continuity equations. Several counterexamples to the uniqueness of continuity equations in a nonsmooth setting are known, see \cite{A, CLR2003, D, ABC1, ABC2} and also \cite{CGSW}, where a similar approach based on convex integration is used. Some of these counterexamples, in particular \cite{D}, can be modified in order to obtain counterexamples to the chain rule with vector fields more regular than $L^{\infty}$. However, these examples rely on explicit constructions and yield only very specific renormalization defects. In particular, diffuse defects and defects which are not measures have not been known previously.
 
We close this introduction with a short comment on our method. We use a convex integration scheme where the perturbations are obtained from laminates, thus taking an approach reminiscent of~\cite{MS, AFS, CS, KRW2, KRW1}. Our convergence strategy relies on Young measures (cf.~\cite{KRW2, KRW1}) and avoids the Baire category method, thus giving a somewhat explicit construction. The core of our proof is a study (in Section~\ref{geom} below) of the geometry of the nonlinear constraint sets $K_C$ (see~\eqref{nonlinear}) in matrix space. It is at this point that the specific properties of our problem enter. Note that in dimension 2 our rank-2 condition would turn into a rank-1 condition, which would be too rigid for the geometric constructions of Section~\ref{geom}.  

\subsection*{Acknowledgments}
This research has been partially supported by the SNSF grants 140232 and 156112. This work was started while the third author was a PostDoc at the Departement Mathematik und Informatik of the Universit\"at Basel. He would like to thank the department for the hospitality and the support. The second author was partially supported by the Russian Foundation for Basic Research, project no.~13-01-12460. The authors are grateful to S.~Bianchini, C.~De Lellis, and L.~Sz\'{e}kelyhidi for the fruitful discussions about the topic of the paper.

\section{Preliminaries}\label{prelim}
A function $\beta:(0,\infty)\to\Rbb$ is called \term{strongly convex} if there exists $\kappa>0$ such that, for all $x_1,x_2>0$ and $0\leq\lambda\leq1$,
\begin{equation}\label{stronglyconvex}
\beta(\lambda x_1+(1-\lambda)x_2)\leq\lambda\beta(x_1)+(1-\lambda)\beta(x_2)-\kappa\lambda(1-\lambda)|x_1-x_2|^2.
\end{equation}
For instance, the map $x\mapsto x^2$ is strongly convex with $\kappa=1$. We remark in passing that for the purposes of this paper, we could replace $|x_1-x_2|^2$ by $\phi(|x_1-x_2|)$, where $\phi:[0,\infty)\to[0,\infty)$ is an increasing function with $\phi(0)=0$ and $\lim_{t\to\infty}\phi(t)=\infty$.
\begin{proposition}\label{strong}
If $\beta:(0,\infty)$ is strongly convex for a $\kappa>0$, and if $\lambda<0$ and $x_1,x_2>0$ are such that $\lambda x_1+(1-\lambda)x_2>0$, then
\begin{equation*}
\lambda\beta(x_1)+(1-\lambda)\beta(x_2)\leq\beta(\lambda x_1+(1-\lambda)x_2)+\kappa\lambda(1-\lambda)|x_1-x_2|^2.
\end{equation*}
\end{proposition} 
\begin{proof}
This follows by replacing $x_1$ by $\lambda x_1+(1-\lambda)x_2$, $x_2$ by $x_1$, and $\lambda$ by $1/(1-\lambda)$ in~\eqref{stronglyconvex}.
\end{proof}
\begin{remark}\label{betanormal}
An immediate remark is that for the proof of Theorem \ref{mainthm} we may assume, without loss of generality, that $\beta(1)=1$. Indeed, by~\eqref{continuity}, equation~\eqref{defect} remains unaffected by adding a constant to $\beta$. We will make this assumption throughout the rest of the paper.
\end{remark}

We recall the space of solenoidal vectorfields on $\Omega$ (cf. Chapter III of~\cite{Gald94MTNS}),  
\begin{equation*}
H(\Omega)=\left\{v\in L^2(\Omega;\Rbb^3):\int_{\Omega}v\cdot\nabla p dx=0\hspace{0.2cm}\text{for every $p\in W^{1,2}(\Omega)$ }\right\}.
\end{equation*}
It is known that if $(v_n)\subset C^1(\bar{\Omega};\Rbb^3)$ is a sequence of divergence-free vector fields such that $v(x)=0$ on $\partial\Omega$, and if the sequence converges weakly in $\Lrm^2(\Omega)$ to a field $v$, then $v\in H(\Omega)$. 

The problem~\eqref{continuity},~\eqref{defect} can then be formulated in the sense of distributions in the following way: Find $v\in H(\Omega)$ such that for every $\psi\in C_c^\infty(\Omega)$, we have
\begin{equation*}
\int_{\Omega}\rho v\cdot\nabla\psi dx=0
\end{equation*}
and
\begin{equation*}
\int_{\Omega}\beta(\rho) v\cdot\nabla\psi dx+\int_{\Omega}f\psi dx=0
\end{equation*}
(if $f$ is merely a distribution, the second integral is of course to be understood as the action of $f$ on $\psi$).

In our iteration scheme, the perturbations will be chosen as members of \term{recovery sequences} of \term{rank-2 laminates}. These are defined as follows (cf.~\cite{Dac1985} and also Definition~9.1 in~\cite{Pedr97PMVP} for the rank-1 analogue):
\begin{definition}\label{rk2laminate}
\begin{itemize}
\item[a)] Suppose $\lambda_i>0$ for $i=1,\ldots,n$, $\sum_{i=1}^{n}\lambda_i=1$, and $U_i\in\Rbb^{3\times3}$ for $i=1,\ldots,n$. The family of pairs $(\lambda_i,U_i)_{i=1}^n$ satisfies the (inductively defined) \term{$H_n$-condition} if
\begin{itemize}
\item[i)] $\rank(U_2-U_1)\leq2$ in the case $n=2$;
\item[ii)] after a relabeling of indices, if necessary, we have $\rank(U_2-U_1)\leq2$ and the family $(\tau_i,V_i)_{i=1}^{n-1}$ satisfies the $H_{n-1}$-condition, where
\begin{equation*}
\begin{aligned}
\tau_1=\lambda_1+\lambda_2, \hspace{0.2cm}\tau_i=\lambda_{i+1}\hspace{0.2cm}\text{for $i=2,\ldots,n-1$}
\end{aligned}
\end{equation*}
and
\begin{equation*}
\begin{aligned}
V_1=\frac{\lambda_1}{\tau_1}U_1+\frac{\lambda_2}{\tau_1}U_2, \hspace{0.2cm}V_i=U_{i+1}\hspace{0.2cm}\text{for $i=2,\ldots,n-1$}
\end{aligned}
\end{equation*}
in the case $n>2$.
\end{itemize}
Moreover we adopt the convention that every pair of the form $(1,U)$ satisfies the $H_1$-condition.

\item[b)] A probability measure $\nu$ on $\Rbb^{3\times3}$ is said to be a \term{rank-2 laminate} of order $n$ if it has the form
\begin{equation*}
\nu=\sum_{i=1}^n\lambda_i\delta_{U_i}
\end{equation*}
for a family $(\lambda_i,U_i)_{i=1}^n$ which satisfies the $H_n$-condition.
\end{itemize}
\end{definition}
For the \term{expectation} of a probability measure, we write
\begin{equation*}
\bar{\nu}:=\int_{\Rbb^{3\times3}}Vd\nu(V).
\end{equation*}
A \term{parametrized probability measure} or \term{Young measure} is a map $\Omega\ni x\mapsto\nu_x$, where $\nu_x$ is a probability measure on $\Rbb^{3\times3}$. It is said to be \term{weakly* measurable} if the map
\begin{equation*}
x\mapsto\int_{\Rbb^{3\times3}}h(z)d\nu_x(z)
\end{equation*}
is measurable in the usual sense for every bounded continuous test function $h:\Rbb^{3\times 3}\to\Rbb$.

We also need to define the \term{rank-2 lamination convex hull} of a set $K\subset\Rbb^{3\times3}$. A similar notion for rank-1 laminates is presented e.g. in Section 4.4 of~\cite{Mull99VMMP}.
\begin{definition}\label{2lc}
Let $K\subset\Rbb^{3\times3}$. A matrix $U\in\Rbb^{3\times3}$ is contained in the \term{rank-2 lamination convex hull} of $K$, denoted $K^{2lc}$, if and only if 
\begin{equation*}
U=\sum_{i=1}^n\lambda_iU_i
\end{equation*} 
for a family $(\lambda_i,U_i)_{i=1}^n$ that satisfies the $H_n$-condition and such that $U_i\in K$ for every $i=1,\ldots,n$.
\end{definition}

\section{Proof of Theorem~\ref{mainthm}}\label{proof}
\textbf{Step 1: Reformulation of the problem.} First we rewrite equations~\eqref{continuity} and~\eqref{defect} as the conjunction of an underdetermined \emph{linear} differential system and a nonlinear \emph{pointwise} constraint, thus adopting a viewpoint similar to the one in~\cite{DLSz2008DI}.

Let us therefore consider the linear system of equations
\begin{equation}\label{linear}
\begin{aligned}
\Div(m)&=0\\
\Div(v)&=0\\
\Div(w)&=f
\end{aligned}
\end{equation} 
in the unknowns $(m,v,w):\Omega\to\Rbb^{3\times3}$. We also define the constraint set, with given constant $C>1$, as
\begin{equation}\label{nonlinear}
\begin{aligned}
K_{C}&:=\left\{(m,v,w)\in\Rbb^{3\times3}: \frac{1}{C}\leq|v|\leq C \right.\\
&\left.\text{ and there is $\frac{1}{C}\leq\rho\leq C$ such that } m=\rho v, w=\beta(\rho)v\right\}.
\end{aligned}
\end{equation}
Thus $K_{C}$ is a non-empty compact subset of $\Rbb^{3\times3}$. Then, clearly, if a triplet of measurable maps $(m,v,w)$ satisfies~\eqref{linear} in the sense of distributions, if $(m,v,w)(x)\in K_{C}$ for almost every $x\in\Omega$, and if $v\in H(\Omega)$, then $v$ and $\rho(x):=|m(x)|/|v(x)|$ will be a solution of~\eqref{continuity} and~\eqref{defect} as in Theorem~\ref{mainthm}.\\

\textbf{Step 2: Recovery of rank-2 laminates.} 
It is convenient to identify a triplet $(m,v,w)$ with the matrix $U$ whose rows are given by $m$, $v$ and $w$. Equations~\eqref{linear} then mean that 
\begin{equation}\label{matrix}
\Div(U)=(0,0,f)^T,
\end{equation}
where the divergence is taken row-wise as usual.

An important building block for our construction is the fact that rank-2 laminates can be approximated in an appropriate sense by solutions of~\eqref{matrix}. This is the content of the following lemma, whose proof is largely standard (c.f. e.g. Proposition~9.2 in~\cite{Pedr97PMVP} or Proposition~19 in~\cite{SzeWie12YMGI} for similar constructions). We give the full proof for the reader's convenience, but postpone it to Section~\ref{recovery}.
\begin{lemma}\label{approx}
Let $K\subset\Rbb^{3\times 3}$ be compact and $(\nu_x)_{x\in\Omega}$ be a weakly*-measurable family of probability measures such that
\begin{itemize}
\item[a)] the measure $\nu_x$ is a rank-2 laminate of finite order for almost every $x\in\Omega$,
\item[b)] $\supp\nu_x\subset K$ for almost every $x$.
\end{itemize}
Assume further that $\psi\in C(\Rbb^{3\times3};\Rbb)$ is a non-negative function that vanishes on $K$. Then the expectation $\bar{\nu}_x$ is well-defined for almost every $x\in\Omega$ and for every $\epsilon>0$ there exists a matrix-valued function $U$ such that
\begin{itemize}
\item[i)] $\Div U=\Div\bar{\nu}$ \hspace{0.2cm} in the sense of distributions,
\item[ii)] \begin{equation*}
\int_{\Omega}\psi(U(x))dx<\epsilon,
\end{equation*}
\item[iii)] \begin{equation*}
\norm{\dist(U(x),K^{2lc})}_{L^{\infty}(\Omega)}<\epsilon,
\end{equation*}
\item[iv)]
\begin{equation}\label{expectationclose}
\int_{\Omega}\left|U(x)-\bar{\nu}_x\right|dx<\int_{\Omega}\int_{\Rbb^{3\times3}}\left|V-\bar{\nu}_x\right|d\nu_x(V)dx+\epsilon.
\end{equation}
\end{itemize}
Moreover, if $\bar{\nu}\in C(\bar{\Omega})$, then $U$ can be chosen to satisfy $U\in C(\bar{\Omega})$ and
\begin{equation*}
U(x)=\bar{\nu}_x\hspace{0.2cm}\text{on $\partial\Omega$.}
\end{equation*}
\end{lemma} 

\textbf{Step 3: Initial step of the iteration.} Our iteration process will start with a triplet of the form $(0,0,w)$, where $\Div(w)=f$. Since our construction is in a sense local, we can ``freeze'' $x$ and first consider a constant vector $w\in\Rbb^3$. The goal is to decompose the matrix $U$ corresponding to $(0,0,w)$ along rank-2 lines as a sum of matrices in $K_{C}$ (of course $K_{C}$ can be viewed as a subset of the space of $3\times3$-matrices). More precisely, we have
\begin{lemma}\label{geom1}
Let $U\in\Rbb^{3\times3}$ such that $U^Te_1=U^Te_2=0$ and $|U^Te_3|\geq1$. Then there exists a rank-2 laminate $\nu=\sum_{i=1}^n\lambda_i\delta_{U_i}$ such that
\begin{equation*}
U=\sum_{i=1}^n\lambda_iU_i
\end{equation*}
and a number $C>1$ such that
\begin{equation*}
\supp\nu\subset K_{C}.
\end{equation*}
Moreover there exists a constant $C_{\beta}$ depending only on $\beta$ such that $C\leq\max\{C_{\beta},4|U^Te_3|\}$.
\end{lemma}
The proof will be given in Section~\ref{geom}.\\

\textbf{Step 4: Subsequent steps of the iteration.} The last lemma we need reads as follows:
\begin{lemma}\label{geom2}
Let $\epsilon>0$ and $\tilde{C}>1$. There exists a strictly increasing continuous function $h:[0,\infty)\to[0,\infty)$, depending only on $\tilde{C}$ and $\beta$, with $h(0)=0$, and a number $\delta>0$, depending only on $\tilde{C}$, $\beta$, and $\epsilon$, such that for every $1<C<\tilde{C}-\epsilon$ and every $U\in\Rbb^{3\times3}$ such that $\dist(U,K^{2lc}_{C})<\delta$, there exists a rank-2 laminate $\nu=\sum_{i=1}^n\lambda_i\delta_{U_i}$ such that
\begin{equation}\label{expectation}
U=\sum_{i=1}^n\lambda_iU_i,
\end{equation}
\begin{equation}\label{L1estimate}
\sum_{i=1}^n\lambda_i|U_i-U|\leq h\left(\dist(U,K_{C})\right),
\end{equation}
and
\begin{equation*}
\supp\nu\subset K_{C+\epsilon}.
\end{equation*}
\end{lemma}
The proof is postponed to Section~\ref{geom}.

\begin{remark}
If $x\mapsto U(x)$ is measurable and satisfies the assumptions of Lemma \ref{geom1} or \ref{geom2} for almost every $x$, respectively, then the laminates $\nu_x$ obtained from the respective lemma form a weakly* measurable family, i.e. a Young measure.  
\end{remark}

\textbf{Step 5: Conclusion.} We are now ready to prove Theorem~\ref{mainthm}. Let $f$ be as in the statement of the theorem. Our goal is to inductively define a sequence $(m_n,v_n,w_n)_{n\geq0}$ of solutions to~\eqref{linear} that approaches the constraint set $K_{C}$ in a suitable sense, for a suitable constants $C>1$.  

First we define the triplet $(m_0,v_0,w_0)$ by setting $v_0\equiv0$, $m_0\equiv0$; $w_0$ is chosen as a bounded continuous solution of $\operatorname{div}w=f$, which exists by assumption. Since the divergence is not affected by adding a constant, we may assume $|w_0(x)|\geq1$ in $\bar{\Omega}$.

Next, let $C_0>1$ be as required by Lemma~\ref{geom1} applied to $U_0(x)$ for \emph{all} $x\in\bar{\Omega}$ (this is possible since $U_0$ is bounded). Next, pick a sequence $(C_n)_{n\geq0}$ that is strictly increasing such that $C_n\nearrow C_0+1=:C$ as $n\to\infty$. We also set $\epsilon_n:=C_{n+1}-C_n$. Then, $(\epsilon_n)$ is a sequence of positive numbers converging to zero.

Identifying $(m_0,v_0,w_0)$ with its corresponding matrix field $U_0$, by Lemma~\ref{geom1} there exists for almost every $x\in\Omega$ a rank-2 laminate $\nu^0_x$ of finite order whose expectation is $U_0(x)$ and whose support is contained in $K_{C_0}$. This completes the definition of $U_0$ and $\nu^0$.

Suppose now that $U_n$ and $\nu^{n}$ have already been constructed for some $n\geq0$ in such a way that $\supp\nu^n\subset K_{C_n}$ and~\eqref{matrix},~\eqref{expectation},~\eqref{L1estimate} are satisfied, that is:
\begin{equation*}
\Div(U_n)=(0,0,f)^T,
\end{equation*}
\begin{equation}\label{nexpectation}
U_{n}(x)=\bar{\nu}_x^n,
\end{equation}
\begin{equation*}
\int_{\Rbb^{3\times3}}|V-U_n(x)|d\nu_x^n(V)\leq h\left(\dist(U_n,K_{C_{n-1}})\right).
\end{equation*}
The last estimate is claimed only for $n\geq1$. By Lemma~\ref{geom2}, where we set $\epsilon=\epsilon_{n+1}$ and $\tilde{C}=C+1$, there exists $\delta_{n+1}=\delta(\epsilon_{n+1})$ such that whenever
\begin{equation*}
\dist(U,K^{2lc}_{C_n})<\delta_{n+1},
\end{equation*}
then there exists a rank-2 laminate whose expectation is $U$ and whose support is contained in
\begin{equation}\label{n+1}
K_{C_n+\epsilon_{n+1}}\subset K_{C_{n+1}}.
\end{equation}
Therefore we apply Lemma~\ref{approx} to $(\nu^n_x)$ with $K_{C_n}$, $\epsilon=\delta_{n+1}$, and 
\begin{equation*}
\psi=h\left(\dist(\frarg,K_{C_n})\right).
\end{equation*}
This yields a matrix field $U_{n+1}$ satisfying 
\begin{equation*}
\Div(U_{n+1})=\Div\left(\bar{\nu}^n_x\right)=\Div(U_n)=(0,0,f)^T,
\end{equation*}
\begin{equation}\label{Kapprox}
\int_{\Omega}h\left(\dist(U_{n+1}(x),K_{C_n})\right)dx<\delta_{n+1},
\end{equation}
and
\begin{equation}\label{distn+1}
\norm{\dist(U_{n+1}(x),K_{C_n}^{2lc})}_{L^{\infty}(\Omega)}<\delta_{n+1}.
\end{equation}
Therefore, by~\eqref{n+1}, we can indeed find, for every $x$, a rank-2 laminate $\nu^{n+1}_x$ with support in $K_{C_{n+1}}$ satisfying~\eqref{expectation} and~\eqref{L1estimate}. This completes the construction of the sequence $(U_n)$.

Next, using~\eqref{nexpectation},~\eqref{expectationclose},~\eqref{L1estimate}, and~\eqref{Kapprox}, we obtain for $n\geq1$
\begin{equation}\label{cauchy}
\begin{aligned}
\int_{\Omega}|U_{n+1}(x)-U_n(x)|dx&=\int_{\Omega}\left|U_{n+1}(x)-\bar{\nu}_x^n\right|dx\\
&\leq\int_{\Omega}\int_{\Rbb^{3\times3}}|V-\bar{\nu}_x^n|d\nu_x^ndx+\delta_{n+1}\\
&\leq\int_{\Omega}h\left(\dist(U_n(x),K_{C_{n-1}})\right)dx+\delta_{n+1}\\
&\leq\delta_n+\delta_{n+1}.
\end{aligned}
\end{equation}

By~\eqref{cauchy} and since we may assume $\delta_n\leq\epsilon_n$, the sequence $(U_n)$ is Cauchy in $\Lrm^1(\Omega)$. Indeed, this follows from $\sum_{n=0}^{\infty}\epsilon_n=C-C_0=1$. Therefore, $(U_n)$ converges strongly in $\Lrm^1$ to a limit matrix field $U_{\infty}\in\Lrm^1(\Omega)$, and up to a subsequence (not relabeled) the convergence even takes place almost everywhere.

Finally, by~\eqref{distn+1} and the observation that $K_{C_n}\subset K_{C}$ for every $n$, the sequence $(U_n)$ is bounded in $\Lrm^{\infty}$, and by~\eqref{Kapprox}
\begin{equation*}
\int_{\Omega}h\left(\dist(U_{n+1}(x),K_{C})\right)dx\leq\int_{\Omega}h\left(\dist(U_{n+1}(x),K_{C_n})\right)dx\to0
\end{equation*}   
as $n\to\infty$. It follows then from dominated convergence that 
\begin{equation*}
\int_{\Omega}h\left(\dist(U_{\infty}(x),K_{C})\right)dx=0,
\end{equation*}
so that $U_{\infty}(x)\in K_{C}$ for almost every $x\in\Omega$. 

As a final observation, since $v_0\equiv0$ and the boundary values of $U_n^Te_2$ remain unchanged in passing from $n$ to $n+1$ thanks to the last statement of Lemma~\ref{approx}, we may conclude $U_{\infty}^Te_2\in H(\Omega)$. According to Step 1, $U_{\infty}$ thus gives rise to the desired solution.\qed

\section{Recovery Sequences for Rank-2 Laminates}\label{recovery}
In this section we prove Lemma~\ref{approx}.

The approximating maps for parametrized measures, whose existece is claimed in the Lemma, will be composed of \term{localized plane waves} as in~\cite{DLSz2008DI}, which satisfy the divergence-free condition
\begin{equation}\label{homo}
\begin{aligned}
\Div(m)&=0\\
\Div(v)&=0\\
\Div(w)&=0.
\end{aligned}
\end{equation}

A \term{plane wave solution} is a solution of~\eqref{homo} of the form $(\bar{m},\bar{v},\bar{w})h(x\cdot\xi)$, where $(\bar{m},\bar{v},\bar{w})\in\Rbb^{3\times3}$ is constant and $\xi\in\Rbb^3\setminus\{0\}$. The function $h:\Rbb\to\Rbb$ is called the \term{profile function}. The \term{wave cone} of~\eqref{homo} is then defined as
\begin{equation*}
\begin{aligned}
\Lambda&:=\left\{(\bar{m},\bar{v},\bar{w})\in\Rbb^{3\times3}:\text{There exists $\xi\neq0$ such that}\right.\\
&\left.\text{$(\bar{m},\bar{v},\bar{w})h(x\cdot\xi)$ satisfies~\eqref{homo} for every smooth $h:\Rbb\to\Rbb$}\right\}.  
\end{aligned}
\end{equation*}
The characterization of the wave cone is standard. To formulate it, it is convenient to identify a triplet $(m,v,w)$ with the matrix $U$ whose rows are given by $m$, $v$ and $w$. Condition~\eqref{homo} then means that 
\begin{equation}\label{homomatrix}
\Div(U)=0,
\end{equation}
where the divergence is taken row-wise as usual.  
\begin{proposition}\label{wavecone}
The wave cone for~\eqref{homomatrix} is the set of all matrices $\bar{U}\in\Rbb^{3\times3}$ whose determinant is zero.
\end{proposition}
\begin{proof}
This follows immediately from the fact that 
$
\Div\left(\bar{U}h(x\cdot\xi)\right)=h'(x\cdot\xi)\bar{U}\xi.
$
\end{proof}
We are now ready to prove Lemma~\ref{approx}, which we recall for convenience:
\begin{lemma*}
Let $K$ be a compact subset of $\Rbb^{3\times3}$, and $(\nu_x)_{x\in\Omega}$ be a weakly*-measurable family of probability measures such that
\begin{itemize}
\item[a)] the measure $\nu_x$ is a rank-2 laminate of finite order for almost every $x\in\Omega$,
\item[b)] $\supp\nu_x\subset K$ for almost every $x$.
\end{itemize}
Assume further that $\psi\in C(\Rbb^{3\times3};\Rbb)$ is a non-negative function that vanishes on $K$. Then, for every $\epsilon>0$ there exists a matrix-valued function $U$ such that
\begin{itemize}
\item[i)] $\Div U(x)=\Div\bar{\nu}_x$ \hspace{0.2cm} for almost every $x\in\Omega$,
\item[ii)] \begin{equation*}
\int_{\Omega}\psi(U(x))dx<\epsilon,
\end{equation*}
\item[iii)] \begin{equation*}
\norm{\dist(U(x),K^{2lc})}_{L^{\infty}(\Omega)}<\epsilon,
\end{equation*}
\item[iv)]
\begin{equation}\label{expectationclose2}
\int_{\Omega}\left|U(x)-\bar{\nu}_x\right|dx<\int_{\Omega}\int_{\Rbb^{3\times3}}\left|V-\bar{\nu}_x\right|d\nu_x(V)dx+\epsilon.
\end{equation}
\end{itemize}
Moreover, if $\bar{\nu}\in C(\bar{\Omega})$, then $U$ can be chosen to satisfy $U\in C(\bar{\Omega})$ and
\begin{equation*}
U(x)=\bar{\nu}_x\hspace{0.2cm}\text{on $\partial\Omega$.}
\end{equation*}
\end{lemma*}
\begin{remark}
In the situation of the Lemma, we say that $U$ approximates the parametrized measure $(\nu_x)$ with precision $\epsilon$.
\end{remark}
\begin{proof}
\textbf{Step 1.} Suppose first that we are dealing with a homogeneous measure with zero expectation, i.e. $x\mapsto\nu_x$ is constant almost everywhere and $\bar{\nu}=0$. To start an inductive argument, consider first the case that $\nu$ is a rank-2 laminate of order 2, i.e. $\nu=\lambda\delta_{U_1}+(1-\lambda)\delta_{U_2}$ with $\rank(U_2-U_1)\leq2$ and $U_1, U_2\in K$. Therefore, by Proposition~\ref{wavecone}, there exists $\xi\in\Rbb^3$ such that the matrix field
\begin{equation*}
U_n(x)=U_1+(U_2-U_1)h(nx\cdot\xi)
\end{equation*}  
is divergence-free for any frequency $n$ and any profile $h$. We choose here as our profile the 1-periodic extension of the function
\begin{equation*}
h(t)=\begin{cases}
1 & \text{if}\hspace{0.2cm} t\in[0,1-\lambda)\\
0 & \text{if}\hspace{0.2cm} t\in[1-\lambda,1).
\end{cases}
\end{equation*}
To achieve zero boundary values, we use a standard cutoff technique as follows: Since $\Div (U_n)=0$, there exists another matrix field $\Phi_n$ such that
\begin{equation*}
U_n=\curl(\Phi_n),
\end{equation*}
the curl being taken row-wise. Moreover, it is not hard to see (e.g. by explicitly writing down a formula for $\Phi_n$) that the potentials $\Phi_n$ may be chosen in such a way that $\norm{\Phi_n}_{L^{\infty}(\Omega)}\to 0$ as $n\to\infty$. As a further remark, observe that $U_n$ (and thus also $\Phi_n$) can be taken smooth by means of a mollification of $h$ with a mollification parameter of size asymptotically $1/n^2$.

For $\delta>0$ let now $\eta_{\delta}\in C_c^\infty(\Omega)$ be a cutoff function such that $0\leq\eta_{\delta}\leq1$ and $\eta_\delta\equiv1$ for all $x\in\Omega$ for which $\dist(x,\partial\Omega)>\delta$. Then, by the product rule,
\begin{equation}\label{inftyestimate}
\norm{\eta_\delta U_n-\curl(\eta_\delta\Phi_n)}_{L^{\infty}(\Omega)}\leq C\norm{\eta_\delta}_{C^1}\norm{\Phi_n}_{L^{\infty}(\Omega)}\leq \frac{C}{\delta}\norm{\Phi_n}_{L^{\infty}(\Omega)},
\end{equation}
so that by choosing, say, $\delta=\delta(n)=\norm{\Phi_n}_{L^{\infty}(\Omega)}^{1/2}$, we can make the left hand side of~\eqref{inftyestimate} arbitrarily small by choosing $n$ sufficiently large. Thus, choosing $U(x)=\curl(\eta_{\delta(n_0)}\Phi_{n_0})$ for a sufficiently large $n_0$, we see that $U$ is as desired: Indeed, i) follows from the fact that $U$ is a curl, the continuity and boundary values follow by construction, iii) is an immediate consequence of~\eqref{inftyestimate} and the fact that $\eta_\delta U_n$ takes values in $K^{2lc}$ for every $x\in\Omega$; properties ii) and iv) are both implied by the observation that the sequence $(\curl(\eta_{\delta(n)}\Phi_{n}))_n$ is uniformly bounded in $L^\infty$ and generates $\nu$ in the sense of Young measures (cf. e.g. Chapter 3 in~\cite{Mull99VMMP}).  

For the induction step, we use the hypothesis that the Lemma be true for laminates of order $n$, and consider a laminate $\nu$ of order $n+1$:
\begin{equation*}
\nu=\sum_{i=1}^{n+1}\lambda_iU_i,
\end{equation*}
where $(\lambda_i,U_i)_i$ satisfies the $H_{n+1}$-condition. Define a laminate of second order by
\begin{equation*}
\tilde{\nu}=\lambda_{n+1}\delta_{U_{n+1}}+(1-\lambda_{n+1})\delta_{\bar{U}}
\end{equation*}
where 
\begin{equation*}
\bar{U}:=\frac{\sum_{i=1}^n\lambda_iU_i}{\sum_{i=1}^n\lambda_i}.
\end{equation*}
Using Definition~\ref{rk2laminate}, it is not hard to see that $\rank(U_{n+1}-\bar{U})\leq2$ and therefore $\tilde{\nu}$ is a rank-2 laminate of second order (we omit the conceivable case that $U_{n+1}=\bar{U}$, which is trivial). We may hence find an approximating map $\tilde{U}$ for $\tilde{\nu}$ with precision $\epsilon$ exactly as in the induction basis (observe that the expectation of $\tilde{\nu}$ is not necessarily zero, which does not matter for our construction however). By construction, the set 
\begin{equation*}
S=\left\{x\in\Omega: \tilde{U}(x)=\bar{U}\right\}
\end{equation*} 
is Lipschitz and we may assume that
\begin{equation*}
\left|\frac{|S|}{|\Omega|}-\sum_{i=1}^n\lambda_{i}\right|<\epsilon.
\end{equation*}
By the induction hypothesis together with Definition~\ref{rk2laminate}, there exists a map $U'$ on $S$ which approximates the measure
\begin{equation*}
\frac{\sum_{i=1}^n\lambda_i\delta_{U_i}}{\sum_{i=1}^n\lambda_i}
\end{equation*}
with precision $\epsilon$. Moreover, $U'=\bar{U}$ on the set $\{x\in S:\dist(x,\partial S)<\delta\}$ for some $\delta>0$. Hence the map defined by
\begin{equation*}
U(x)=\begin{cases}
U'(x) &\hspace{0.2cm}\text{if $x\in S$}\\
\tilde{U}(x) &\hspace{0.2cm}\text{if $x\in\Omega\setminus S$}
\end{cases}
\end{equation*}
is smooth and satisfies the requirements of the Lemma.

\textbf{Step 2.} As a next step, consider a possibly non-homogeneous measure $(\nu_x)_x$, whose expectation $\bar{\nu}$ is however still assumed to be identically zero. This case can be treated as usual by approximating $\nu$ by a piecewise homogeneous measure and applying Step 1 to each piece. For details see e.g. Section 4.9 in~\cite{Mull99VMMP}. Observe that, in this step, we may even allow $K$ to depend on $x\in\Omega$ (in a measurable fashion).

\textbf{Step 3.} Let now $(\nu_x)_x$ be of full generality as assumed in the Lemma. Consider the \term{shifted} measure $\mu_x$ defined by duality via
\begin{equation*}
\int_{\R^{d\times d}}h(z)d\mu_x(z)=\int_{\R^{d\times d}}h(z-\bar{\nu}_x)d\nu_x(z)
\end{equation*}
for a.e. $x\in\Omega$ and every test function $h\in C_b(\R^{d\times d})$. Then one sees easily that $\mu_x$ is still a rank-2 laminate, and moreover for its expectation $\bar{\mu}_x$ we have 
\begin{equation*}
\bar{\mu}_x=0\hspace{0.3cm}\text{for a.e. $x\in\Omega$.}
\end{equation*}
Applying Step 2 to $\mu$ with $K$ replaced by $K-\bar{\nu}_x$ (cf. the last observation in Step 2) yields an approximating map $W$ for $\mu$. One can then easily check that 
\begin{equation*}
U:=W+\bar{\nu}
\end{equation*}
approximates $\nu$ in the sense of the Lemma.
\end{proof}

\section{Geometry of the Nonlinear Constraint}\label{geom}
\subsection{Proof of Lemma~\ref{geom1}}
In this subsection we prove the first geometric lemma, which we recall for the reader's convenience:
\begin{lemma*}
Let $U\in\Rbb^{3\times3}$ such that $U^Te_1=U^Te_2=0$ and $|U^Te_3|\geq1$. Then there exists a rank-2 laminate $\nu=\sum_{i=1}^n\lambda_i\delta_{U_i}$ such that
\begin{equation*}
U=\sum_{i=1}^n\lambda_iU_i
\end{equation*}
and a number $C>1$ such that
\begin{equation*}
\supp\nu\subset K_{C}.
\end{equation*}
Moreover there exists a constant $C_\beta$ depending only on $\beta$ such that $C\leq\max\{C_\beta,4|U^Te_3|\}$.
\end{lemma*}
\begin{proof}
Let $U$ be as in the statement of the lemma. As usual, we identify it with the triplet $(m,v,w)$ of its row vectors, so by assumption, $m=v=0$ and $|w|\geq1$. We split $(0,0,w)$ into
\begin{equation*}
(0,0,w)=\frac{1}{2}\left(-w,-w,w\right)+\frac{1}{2}\left(w,w,w\right).
\end{equation*}  
If we call the matrices correponding to the two triplets on the right hand side $U_-$ and $U_+$, respectively, we first observe that $U_-$ and $U_+$ are rank-2 connected since $(U_--U_+)e_3=0$. Secondly, $U_+\in K_{C}$ for any $C$ such that
\begin{equation*}
C\geq|w|
\end{equation*} 
(recall $\beta(1)=1$ by Remark~\ref{betanormal}).

Next, let us further decompose $U_-$. We make the ansatz 
\begin{equation}\label{decomp}
\left(-w,-w,w\right)=\frac{1}{2}(\rho_1v_1,v_1,\beta(\rho_1)v_1)+\frac{1}{2}(\rho_2v_2,v_2,\beta(\rho_2)v_2)
\end{equation}
with
\begin{equation}\label{decomp2}
v_1=w,\hspace{0.2cm}v_2=-3w.
\end{equation}
Then clearly~\eqref{decomp} is a rank-2 decomposition (in fact even rank-1), and~\eqref{decomp} and~\eqref{decomp2} result in the conditions
\begin{equation}\label{rhosystem}
\begin{aligned}
-\rho_1+3\rho_2&=2\\
-\beta(\rho_1)+3\beta(\rho_2)&=-2.
\end{aligned}
\end{equation}
Let us show that these equations can be satisfied thanks to the strong convexity assumption on $\beta$. Indeed, suppose $-\rho_1+3\rho_2=2$. Then, using Proposition~\ref{strong}, we calculate
\begin{equation}\label{convexestimate}
\begin{aligned}
-\beta(\rho_1)+3\beta(\rho_2)&=2\left(-\frac{1}{2}\beta(\rho_1)+\frac{3}{2}\beta(\rho_2)\right)\\
&\leq2\beta\left(-\frac{1}{2}\rho_1+\frac{3}{2}\rho_2\right)-2\kappa\frac{3}{4}|\rho_1-\rho_2|^2\\
&=2-\frac{3}{2}\kappa|\rho_1-\rho_2|^2.
\end{aligned}
\end{equation}
Finally, the equation $-\rho_1+3\rho_2=2$ can be rewritten as $\rho_1-\rho_2=2\rho_2-2$, and therefore by~\eqref{convexestimate} we can achieve~\eqref{rhosystem} by choosing $\rho_2>1$ sufficiently large and then setting $\rho_1=3\rho_2-2>1$. 

Since, with this choice of $\rho_1,\rho_2$, the triplets $(\rho_1v_1,v_1,\beta(\rho_1)v_1)$ and $(\rho_2v_2,v_2,\beta(\rho_2)v_2)$ are in $K_{C}$ for a suitable $C$, the proof is finished. In particular, the estimate for $C$ in the statement of the lemma follows directly from our construction.
\end{proof}

\subsection{Proof of Lemma~\ref{geom2}} Recall Lemma~\ref{geom2}:
\begin{lemma*}
Let $\epsilon>0$ and $\tilde{C}>1$. There exists a strictly increasing continuous function $h:[0,\infty)\to[0,\infty)$, depending only on $\tilde{C}$ and $\beta$, with $h(0)=0$, and a number $\delta>0$, depending only on $\tilde{C}$, $\beta$, and $\epsilon$, such that for every $1<C<\tilde{C}-\epsilon$ and every $U\in\Rbb^{3\times3}$ such that $\dist(U,K^{2lc}_{C})<\delta$, there exists a rank-2 laminate $\nu=\sum_{i=1}^n\lambda_i\delta_{U_i}$ such that
\begin{equation}\label{expectation2}
U=\sum_{i=1}^n\lambda_iU_i,
\end{equation}
\begin{equation}\label{L1estimate2}
\sum_{i=1}^n\lambda_i|U_i-U|\leq h\left(\dist(U,K_{C})\right),
\end{equation}
and
\begin{equation}\label{suppestimate}
\supp\nu\subset K_{C+\epsilon}.
\end{equation}
\end{lemma*}
\begin{proof}
As usual we denote by $(m,v,w)$ the rows of the matrix $U$. We proceed in five steps:

\textbf{Step 1.} Suppose the vectors $(m,v,w)$ are collinear, so that there exist real numbers $\alpha$, $\gamma$ such that $m=\alpha v$ and $w=\gamma v$. Note that if $\delta'$ is sufficiently small, then $\dist(U,K_{C})<\delta'$ implies
\begin{equation}\label{vestimate}
\frac{1}{C+\epsilon}<|v|<C+\epsilon.
\end{equation}
Note that the meaning of ``sufficiently small'' here can be understood to depend only on $\epsilon$ and $\tilde{C}$.
We want to find a decomposition using the ansatz
\begin{equation*}
(m,v,w)=\lambda(m_1,v_1,w_1)+(1-\lambda)(m_2,v_2,w_2),
\end{equation*} 
where $v_1=\tau_1v$ and $v_2=\tau_2v$. Clearly, this defines a rank-2 (even rank-1) decomposition regardless of the values of $\lambda$, $\tau_1$ and $\tau_2$. The requirement that $(m_1,v_1,w_1)$ and $(m_2,v_2,w_2)$ lie in the set $K_{C+\epsilon}$ then leads to the requirement that there exist $\rho_1,\rho_2>0$ such that
\begin{equation}\label{tausystem}
\begin{aligned}
\lambda\tau_1+(1-\lambda)\tau_2&=1\\
\lambda\tau_1\rho_1+(1-\lambda)\tau_2\rho_2&=\alpha\\
\lambda\tau_1\beta(\rho_1)+(1-\lambda)\tau_2\beta(\rho_2)&=\gamma.
\end{aligned}
\end{equation}
If it happens that $\gamma=\beta(\alpha)+\eta$ for some $\eta\geq0$, we set $\tau_1=\tau_2=1$ so that the first equation of~\eqref{tausystem} is automatically satisfied and the other two equations become
\begin{equation}\label{etasystem1}
\begin{aligned}
\lambda\rho_1+(1-\lambda)\rho_2&=\alpha\\
\lambda\beta(\rho_1)+(1-\lambda)\beta(\rho_2)&=\beta(\alpha)+\eta.
\end{aligned}
\end{equation} 
By the first of these equations and the strong convexity of $\beta$, we have
\begin{equation*}
\lambda\beta(\rho_1)+(1-\lambda)\beta(\rho_2)\geq\beta(\alpha)+\kappa\lambda(1-\lambda)|\rho_1-\rho_2|^2.
\end{equation*}
Therefore, it is possible to find functions $\lambda(\eta)$, $\rho_1(\eta)$, and $\rho_2(\eta)$, depending on $\beta$ and $\alpha$, that are continuous in $\eta$ and satisfy $\lambda(0)=1$, $\rho_1(0)=\rho_2(0)=\alpha$ such that~\eqref{etasystem1} is satisfied for every $\eta\geq0$. Since, if $\dist(U,K_{C})<\delta'$, we can make $\eta$ arbitrarily small by choosing $\delta'$ sufficiently small (depending only on $\tilde{C}$, $\beta$, and $\eta$), we can ensure 
\begin{equation*}
\frac{1}{C+\epsilon}<\rho_1,\rho_2<C+\epsilon
\end{equation*}
for $\delta'$ small enough. Together with~\eqref{vestimate} we conclude that 
\begin{equation*}
(m_1,v_1,w_1),(m_2,v_2,w_2)\in K_{C+\epsilon}.
\end{equation*} 
Thus we have established~\eqref{expectation2} and~\eqref{suppestimate}. 

Next, suppose $\gamma=\beta(\alpha)-\eta$ for some $\eta>0$. Then, in~\eqref{tausystem} we choose $\tau_1=(2-\lambda)/\lambda$ and $\tau_2=-1$ to eliminate the first equation and arrive at  
\begin{equation}\label{etasystem2}
\begin{aligned}
\lambda\tau_1\rho_1-(1-\lambda)\rho_2&=\alpha\\
\lambda\tau_1\beta(\rho_1)-(1-\lambda)\beta(\rho_2)&=\beta(\alpha)-\eta.
\end{aligned}
\end{equation} 
Then, by Proposition~\ref{strong} (replacing $\lambda$ by $-(1-\lambda)$, $x_1$ by $\rho_2$ and $x_2$ by $\rho_1$ and keeping in mind $\lambda\tau_1-(1-\lambda)=1$), we have
\begin{equation*}
\lambda\tau_1\beta(\rho_1)-(1-\lambda)\beta(\rho_2)\leq\beta(\alpha)-\kappa(2-\lambda)(1-\lambda)|\rho_1-\rho_2|^2.
\end{equation*}
Assertions~\eqref{expectation2} and~\eqref{suppestimate} then follow by the same arguments as above, observing that again $\eta=0$ corresponds to $\lambda=1$, $\rho_1=\rho_2=\alpha$. This completes Step 1. 

Notice again that $\delta'>0$ constructed in Step 1 depends on $\epsilon$, $\beta$, and $\tilde{C}$, but not on $U$, or $C$.

\textbf{Step 2.} Suppose now that $m$ and $v$ are parallel, that is, there exists a real number $\alpha$ such that $m=\alpha v$. (We are no longer assuming that $w$ be parallel with $m$ and $v$.) Again, we wish to represent $(m,v,w)$ as a rank-2 combination of two triplets,
\begin{equation*}
(m,v,w)=\lambda(m_1,v_1,w_1)+(1-\lambda)(m_2,v_2,w_2),
\end{equation*}
where $m_i$, $v_i$, $w_i$ are collinear ($i=1,2$), so that we can proceed as in Step 1. To this end, take the ansatz $m_i=\alpha v_i$, $w_i=\mu_iv_i$, and set $\lambda=1/2$:
\begin{equation}\label{musystem}
\begin{aligned}
v_1+v_2&=2v\\
\alpha v_1+\alpha v_2&=2\alpha v\\
\mu_1v_1+\mu_2v_2&=2w.
\end{aligned}
\end{equation}
First, clearly $(m_2,v_2,w_2)-(m_1,v_1,w_1)$ has rank at most 2 with this ansatz. Secondly, if $v_1$ and $v_2$ are chosen linearly independent and in the plane spanned by $v$ and $w$, then they form a basis of this subspace and therefore~\eqref{musystem} can be solved (if $w$ and $v$ are already parallel, it can be trivially solved). More specifically, if $\eta>0$, then by choosing $\delta''>0$ small enough (depending only on $\tilde{C}$, $\beta$, and $\eta$) we can ensure that $\dist(U,K_{C})<\delta''$ implies $|w-\beta(\alpha)v|<\eta$. When $w=\beta(\alpha)v$ exactly, we can simply set $v_1=v_2=v$ and $\mu_1=\mu_2=\beta(\alpha)$. Therefore, there exist continuous maps $v_i(w)$, $\mu_i(w)$ ($i=1,2$) depending on $\alpha$, $\beta$ such that $v_i(\beta(\alpha)v)=v$ and $\mu_i(\beta(\alpha)v)=\beta(\alpha)$ and such that~\eqref{musystem} is satisfied for any $w$. It follows that, by choosing $\delta''>0$ sufficiently small, $\dist(U,K_{C})<\delta''$ guarantees 
\begin{equation*}
\dist((m_i,v_i,w_i),K_{C})<\delta'\hspace{0.2cm}(i=1,2)
\end{equation*}  
for the number $\delta'$ established in Step 1. We may therefore decompose each $(m_i,v_i,w_i)$ further as in Step 1, which yields a rank-2 decomposition of $(m,v,w)$ into (at most) four triplets in $K_{C+\epsilon}$, each satisfying~\eqref{expectation2},~\eqref{L1estimate2}. Note again that $\delta''$ depends only on $\tilde{C}$, $\beta$, and $\epsilon$.

\textbf{Step 3.} Consider now a general triplet $(m,v,w)$. We want to decompose $(m,v,w)$ into two triplets along rank-2 lines,  
\begin{equation*}
(m,v,w)=\lambda(m_1,v_1,w_1)+(1-\lambda)(m_2,v_2,w_2),
\end{equation*}
such that there exist $\alpha_1,\alpha_2$ such that $m_1=\alpha_1v_1$, $m_2=\alpha_2v_2$, so that Step 2 can be applied to both $(m_i,v_i,w_i)$ individually. We take the ansatz $\lambda=1/2$, $w_1=w_2=w$ (thereby ensuring our decomposition runs along a rank-2 line), to obtain the equations
\begin{equation*}
\begin{aligned}
v_1+v_2&=2v\\
\alpha_1v_1+\alpha_2v_2&=2m.
\end{aligned}
\end{equation*} 
The exact same reasoning as in Step 2 then yields a $\delta>0$ depending only on $\tilde{C}$, $\beta$, and $\epsilon$ such that $\dist((m,v,w),K_{C})<\delta$ ensures that
\begin{equation*}
\dist((m_i,v_i,w_i),K_{C})<\delta''\hspace{0.2cm}(i=1,2),
\end{equation*} 
where $\delta''$ is the number from Step 2.

\textbf{Step 4.} So far we have produced $\delta>0$ such that the assertions of Lemma~\ref{geom2} are true provided $\dist(U,K_{C})<\delta$. Let now $U$ be such that only
\begin{equation*}
\dist(U,K^{2lc}_{C})<\delta.
\end{equation*}
By assumption and the definition of the rank-2 lamination convex hull (Definition~\ref{2lc}), $U$ can be written as
\begin{equation*}
U=\sum_{i=1}^n\lambda_iU_i+\tilde{U}=\sum_{i=1}^n\lambda_i(U_i+\tilde{U}),
\end{equation*}
where $|\tilde{U}|<\delta$, the family $(\lambda_i,U_i)$ satisfies the $H_n$-condition, and $U_i\in K_{C}$ ($i=1\ldots n$). But for every $i$, we can now apply Steps 1--3 to $U_i+\tilde{U}$, which completes the proof of Lemma~\ref{geom2} modulo the estimate~\eqref{L1estimate2}.

\textbf{Step 5.} It remains to exhibit a function $h$ that renders~\eqref{L1estimate2} correct. To this end, recall that the $\lambda_i$ and $U_i$ which we constructed in the previous steps depended solely on $U$ and $\beta$, so that in particular the left hand side of~\eqref{L1estimate2} is independent of $C$. Moreover, if $U\in K_{\tilde{C}}$, our construction leaves $U$ unchanged, so that the left hand side of~\eqref{L1estimate2}, considered as a function of $U$ (with $\beta$ fixed), is zero on $K_{\tilde{C}}$. The last observation needed is that, by construction, the left hand side $\sum_{i=1}^n\lambda_i|U_i-U|$ depends on $U$ continuously in a $\delta$-neighborhood of $K_{\tilde{C}}^{2lc}$.

The distance function $\dist(U,K_{\tilde{C}})$ is of course zero on $K_{\tilde{C}}$ and positive elsewhere (since $K_{\tilde{C}}$ is compact). Therefore, we may define
\begin{equation*}
h(t)=\max_{\mathcal{U}_t}\left\{\sum_{i=1}^n\lambda_i|U_i-U|\right\},
\end{equation*} 
where we set $\mathcal{U}_t=\{U\in\Rbb^{3\times3}:\dist(U,K_{\tilde{C}})=t\}$. Again we considered the left hand side of~\eqref{L1estimate2} as a continuous function of $U$. We may further assume $h$ to be strictly increasing by choosing it larger if necessary.

Then, by definition of $h$ we have
\begin{equation*}
\sum_{i=1}^n\lambda_i|U_i-U|\leq h\left(\dist(U,K_{\tilde{C}})\right)\leq h\left(\dist(U,K_{C})\right)
\end{equation*}
for any $C\leq\tilde{C}$, since then $K_{{C},}\subset K_{\tilde{C}}$. The proof is thus complete. 
\end{proof}

\bibliography{defect}
\bibliographystyle{amsalpha}








\end{document}